\documentclass[12pt]{amsart}

\usepackage{amsfonts}
\usepackage{mathrsfs}
\usepackage{amsmath}
\usepackage{amssymb}
\usepackage{amsthm}
\usepackage{enumitem}
\usepackage{latexsym}
\usepackage{color}
\usepackage{geometry}
\theoremstyle{definition}
\newtheorem{defin}{Definition}[section]
\newtheorem{prop}[defin]{Proposition}
\newtheorem{lem}[defin]{Lemma}
\newtheorem{exam}[defin]{Example}

\newtheorem{cor}[defin]{Corollary}

\newtheorem{remark}[defin]{Remark}

\newtheorem{question}[defin]{Question}

\newtheorem{notation}[defin]{Notation}

\newtheoremstyle{TheoremNum}
{\topsep}{\topsep}              
{\itshape}                      
{}                              
{\bfseries}                     
{.}                             
{ }                             
{\thmname{#1}\thmnote{ \bfseries #3}}
\theoremstyle{TheoremNum}

\DeclareMathOperator{\calS}{{\mathcal S}}
\DeclareMathOperator{\Ult}{Ult}
\DeclareMathOperator{\RK}{RK}

\DeclareMathOperator{\Lim}{\mathbf{Lim}}

\DeclareMathOperator{\supp}{ \text{supp }}

\newcommand{\comment}[1]{}

\begin{document}

	\title[Countable compactness of finite powers of a Wallace semigroup]{A Wallace semigroup whose every finite power is countably compact}
	\author[J. L. J. Fuentes-Maguiña]{Juan Luis Jaisuño Fuentes-Magui\~na}
 	\author[V.  O. Rodrigues]{Vinicius de Oliveira Rodrigues}

	\author[A. H. Tomita]{Artur Hideyuki Tomita}
	\address{Departamento de Matemática, Instituto de Matemática e Estatística, Universidade de São Paulo, Rua do Mat\~ao, 1010 -- CEP 05508-090, S\~ao Paulo, SP - Brazil}
	\email{juanfm@ime.usp.br, vinior@ime.usp.br, tomita@ime.usp.br}

	\subjclass{Primary 54D20, 54H11, 22A15; Secondary 54A35, 54G20.}
	\date{}
	\commby{}
	\keywords{Wallace semigroup, topological group, countable compactness, selective ultrafilter, free Abelian group}

	\begin{abstract}
	We show that, assuming the existence of $\mathfrak{c}$ incomparable selective ultrafilters, there exists a Wallace semigroup whose infinite countable power is the least power which fails to be countably compact. This answers positively Question 9.4 of \cite{Tomita15}.  
	\end{abstract}
	
	\maketitle

\section{Introduction}

  \subsection{A bit of history}

    It is a classical and well-known result that compact two-sided cancellative  semigroups are topological groups (e.g. \cite{gelbaum1951embedding}). 
    In 1955, Wallace \cite{W} has asked if every countably compact two-sided cancellative semigroup is a topological group. This question has been known as the Wallace's problem, and an example yielding a negative answer is known as a Wallace semigroup. The first consistent example of such a semigroup, due to D. Robbie and S. Svetlichny, appeared forty-one years later in the Proceedings of the American Mathematical Society \cite{RS}, under the Continuum Hypothesis (CH). Since then, several examples of Wallace semigroups have been constructed from Martin's axiom and selective ultrafilters.

   Grant \cite{Grant} showed that sequentially compact two-sided cancellative semigroups are topological groups and asked whether $p$-compact two-sided cancellative semigroups are topological groups.    Recall that a topological space is $p$-compact for some free ultrafilter $p$ if and only if its $2^\mathfrak c$-th power is countably compact. Taking that into account, in \cite{Grant}, Grant also asked whether every two-sided cancellative topological semigroup with countably compact square is a topological group. The first question was answered positively in \cite{T96}, in ZFC. A consistent counter-example for the second question was given in \cite{Boero}. Thus, it natural to ask what is the highest degree of countably compactness which still admits a Wallace semigroup.
    
    In \cite{Boero}, Boero and Tomita consistently produced a group topology for $\mathbb Z^{(\mathfrak c)}$ whose square is countably compact using $\mathfrak c $ selective ultrafilters and, as a corollary, they showed that its subsemigroup $\mathbb N^{(\mathfrak c)}$ with the subspace topology was a Wallace semigroup whose square was still countably compact.
    
    Moreover, in \cite{Tomita15}, Tomita consistently proved that $\mathbb Z^{(\mathfrak c)}$ admits a group topology whose all finite powers are countably compact. For that, the existence of $\mathfrak{c}$ pairwise incomparable selective ultrafilters was assumed - an hypothesis weaker than both the Continuum Hypothesis and Martin's Axiom. The author tried to consider the subsemigroup $\mathbb N^{(\mathfrak c)}$, but he was unable to guarantee that its finite powers are countably compact as well. With that in mind, he asked whether there exists a Wallace semigroup whose all finite powers are countably compact (Problem 9.4.).
    
    In this paper, we consistently answer the preceding question positively by showing that,assuming the existence of $\mathfrak c$ incomparable selective ultrafilters, the semigroup $\mathbb N\oplus \mathbb Z^{(\mathfrak c)}$ can be endowed with a Hausdorff topology whose all finite powers are countably compact. We also show (in ZFC) that this semigroup cannot have a topology whose $\omega$-th power is countably compact. Note that this semigroup is also a subsemigroup of $\mathbb Z^{(\mathfrak c)}$.
    
    During the development of this paper, we also could not make an example out of $\mathbb N^{(\mathfrak c)}$. Thus, the existence of a semigroup topology in $\mathbb N^{(\mathfrak c)}$ whose every finite power is countably compact remains an open problem.

\subsection{Basic results and notation} 

    In this section, we mention some basic results and concepts used in this paper and fix some basic notation as well.
    
    $\Lim$ denotes the set of limit ordinals in $\mathfrak c$, the cardinality of the continuum.
    
    $\mathbb N$ denotes the subset of non-positive integers and $\omega$ is the first limit ordinal. These objects can, of course, be identified, but we reserve the symbol $\mathbb N$ for this specific use for clarity. We consider that $0\in \mathbb N\subseteq \mathbb Z\subseteq \mathbb Q$.

    For any object $a$, $\vec {a}$ denotes the unique sequence of constant value $a$. 
 
    Let $X$ be a non-empty set and $G$ be an Abelian group with unitary element $0$. Given a function $f : X \to G$, the support of $f$ is $\supp{f} = \lbrace x \in G : f(x) \neq 0 \rbrace$. The Abelian group $\lbrace f \in G^X : \vert \supp{f} \vert < \omega \rbrace$ will be denoted by  $G^{(X)}$. When $X$ is clear from the context, for any subset $Y \subset X$, $R^{(Y)}$ denotes the subgroup $\lbrace f \in G^{(X)} \mid \supp{f} \subset Y \rbrace$. If $\bar f$ is a sequence in $G^X$, $\supp \bar f=\bigcup_{i \in \omega} \supp f(i)$.

    Let $p$ be an ultrafilter. $\Ult_p(G)$ is the quotient $G^\omega/\simeq_p$, where $\simeq_p$ is the equivalence relation on $X^\omega$ defined by $f\simeq_p g$ if and only if $\{n \in \omega \mid f(n)=g(n)\} \in p$. $\Ult_p(G)$  is easily seen to have a natural Abelian group structure. Moreover, if $G$ is, in fact, a $R$-module for some commutative ring with unit $R$, then $\Ult_p(G)$ has a natural $R$-module structure.
    
    Given $f \in X^{\omega}$, $[f]_p$ is the equivalence class determined by $\simeq_p$. The mapping that associated $g \in G$ to $[\vec g]_p\in \Ult_p(G)$ is easily seen to be a group monomorphism. We say that an element of $\Ult_P(G)$ is constant if it is in the range of this monomorphism, that is, if it is of the form $[\vec g]_p$, with $g \in G$. The set of all constants of $\Ult_p(G)$ is denoted by $\underline{G}$. Notice that in the case that $G$ is also a $R$-module, the preceding mapping is in fact an $R$-module monomorphism.

    The following shorthand will come in handy:

        \begin{defin}
        Let $G$ be an Abelian group, $p$ be an ultrafilter, and $(f_i: i \in I)$ be a family of members of $G^\omega$.  We say that $(f_i: i \in I)$ is $p$-independent mod constants if:
        \begin{enumerate}
            \item the mapping $i\rightarrow [f_i]_p+\underline G\in \Ult_p(G)/\underline G$ is injective, and
            \item $\{[f_i]_p+\underline G: i \in I\}\subseteq \Ult_p(G)/\underline G$ is $\mathbb Z$-linearly independent.
        \end{enumerate}
    \end{defin} 

    We will be working with subspaces of the vector space $\mathbb Q^{(\mathfrak c)}$. In this context, for every $\alpha \in \mathfrak c$, $\chi_\alpha\in N^{(\mathfrak c)}$ denotes the characteristic function of $\{\alpha\}$, so $\chi_\alpha(\alpha)=1$ and $\chi_\alpha(\beta)=0$ for all $\beta\in \mathfrak c\setminus \{\alpha\}$.
    
    Given a topological space $X$, a sequence $f:\omega\rightarrow X$ and a free ultrafilter $p$, we say a point $x \in X$ is a $p$-limit of $f$ if for every neighborhood $U$ of $x$, $\{n \in \omega: f(n)\in U\}\in p$. If $X$ is Hausdorff, every sequence has at most one $p$-limit, and, when it exists, we denote it by $p$-$\lim$ $f$.  We say that a topological space is $p$-compact if every sequence has a $p$-limit.

    A countably compact space is a topological space such that every countable open cover has a finite subcover - or, equivalently, if every sequence in the space has an accumulation point. It is well known that every accumulation point of a sequence is also a $p$-limit of such sequence, for some ultrafilter $p$.

    As already mentioned, a Wallace semigroup is a two-sided cancellative countably compact topological semigroup which is not a topological group.
    
    A free ultrafilter $p$ is said to be selective (or Ramsey) if for every $f:[\omega]^2\rightarrow 2$ there exists $A \in p$ such that $f|_{[A]^2}$ is constant. It is well known that selective ultrafilters are exactly the minimal ultrafilter in the Rudin-Keisler order, which is the pre-order defined in the collection of all free ultrafilters defined by  $p \leq_{\RK} q$ if there exists a function $f : \omega \to \omega$ such that $p=\{f^{-1}[A]:A\in q\}$. Two selective ultrafilters are said to be incomparable if they are incomparable with respect to $\leq_{\RK}$. The existence of selective ultrafilters is known to be independent from ZFC. 

    The group $\mathbb T$ is the unit circle, which can be defined as the quotient $\mathbb R/\mathbb Z$.

\section{A briefing on stacks}

    In \cite{Tomita15}, a structure called \textit{stack} has been defined. In this section, we quickly review the main results regarding this structure. This section is not intended to be a reference on stacks - we only mention what is strictly necessary to make this paper self-contained. 
    
    Roughly speaking, a stack is a finite collection of sequences in $\mathbb Z^{(\mathfrak c)}$, along with an infinite set of indexes $A\subseteq \mathbb N$ and an integer $N$ (along with some other auxiliary objects) satisfying certain properties that helps to generate homomorphisms with accumulation and convergence properties of interest. In \cite{Tomita15}, Tomita proved that every collection of sequences in $\mathbb Z^{(\mathfrak c)}$ can be associated to an stack so that if $A$ is in some free ultrafilter $p$ and the sequences in this stack have $p$-limits in a group topology, so does the original collection. 
    
        Below we give the complete definition of an integer stack. We will not discuss all the details of the definition. The reader interested in why each of the bullet points is needed may refer to \cite{Tomita15}. We do not recommend trying to figure out the importance of each bullet point solely after reading the definition, and understanding that is not important for understanding the proofs in this paper.
       
    \begin{defin} 
    Let $A$ be an infinite subset of $\omega$. An integer stack $\mathcal S$ on $A$ is composed by
		\begin{itemize}
			\item natural numbers $s,t,M$, positive integers $r_i$ for each $0 \leq i < s$, and positive integers $r_{i,j}$ for each $0 \leq i < s$ and $0 \leq j < r_i$,
			\item functions $f_{i,j,k} \in ( \mathbb Z^{(\mathfrak c)})^{A}$ for each $0 \leq i < s$, $0 \leq j < r_i$, $0 \leq k < r_{i,j}$ and $g_l \in ( \mathbb Z^{(\mathfrak c)})^{A}$ for each $0 \leq l < t$, 
			\item sequences $\xi_i \in \mathfrak c^{A}$ for each $0 \leq i < s$ and $\mu_l \in \mathfrak c^{A}$ for $0 \leq l < t$, 
			\item real numbers $\theta_{i,j,k}$ for each $0 \leq i < s$, $0 \leq j < r_i$, $0 \leq k < r_{i,j}$,
		\end{itemize}
	satisfying the following
		\begin{enumerate}[label=\roman*)]
			\item $\mu_l(n) \in \supp g_l(n)$ for every $n \in A$,
            \item $\mu_{l^*}(n) \notin \supp g_l(n)$ for each $n \in A$ and $0 \leq l^* < l < t$,
			\item  the elements of $\lbrace \mu_l(n) \mid 0 \leq l < t, n \in A \rbrace$ are pairwise distinct,
			\item $\vert g_l(n) \vert \leq M$ for each $n \in A$ and $0 \leq l < t$,
			\item $(\theta_{i,j,k} \mid 0 \leq k < r_{i,j} )$ is injective and linearly independent (as a $\mathbb Q$-vector space) for each $0 \leq i < s$ and $0 \leq j < r_i$,
            \item $\left( \frac{f_{i,j,k}(n)(\xi_i(n))}{f_{i,j,0}(n)(\xi_i(n))} \right)_{n \in A} \to \theta_{i,j,k}$ for each $0 \leq i < s$, $0 \leq j < r_i$ and $0 \leq k < r_{i,j}$,
			\item $\left( \vert f_{i,j,k}(n)(\xi_i(n)) \vert \right)_{n \in A}$ converges monotonically to $+\infty$ for each $0 \leq i < s$, $0 \leq j < r_i$ and $0 \leq k < r_{i,j}$,
			\item $\vert f_{i,j,k}(n)(\xi_i(n)) \vert > \vert f_{i,j,k^*}(n)(\xi_i(n)) \vert$ for each $n \in A$, $i < s$, $j < r_i$ and $0 \leq k < k^* < r_{i,j}$,
			\item $\left( \frac{f_{i,j,k}(n)(\xi_i(n))}{f_{i,j^*,k^*}(n)(\xi_i(n))} \right)_{n \in A}$  converges monotonically to $0$ for each $0 \leq i < s$, $0 \leq j^* < j < r_i$, $0 \leq k < r_{i,j}$ and $0 \leq k^* < r_{i,j^*}$,
			\item $\lbrace f_{i,j,k}(n)(\xi_{i^*}(n)) \mid n \in A \rbrace \subset [-M, M]$ for each $0 \leq i^* < i < s$, $0 \leq j < r_i$ and $0 \leq k < r_{i,j}$.
		\end{enumerate}
    The support of an integer stack $\mathcal S$ on $A$ is defined by the union of the support of each $f_{i,j,k}(n)$ and $g_l(n)$ (for $n \in A$). 
	\end{defin}
     
    \begin{defin}
    Let $\mathcal S$ be an integer stack on $A$ and $N$ be a positive integer. The $N$th root of $\mathcal S$, denoted by $\frac{1}{N}. \mathcal S$, is obtained by replacing $f_{i,j,k}$ by $\frac{1}{N}. f_{i,j,k}$ for each $i < s$, $j < r_i$ and $k < r_{i,j}$, replacing $g_l$ by $\frac{1}{N}. g_{l}$ for each $l < t$, and keeping the rest of the structure of $\mathcal{S}$. 

    A \textit{stack} is a structure of the form $\frac1N \mathcal S$, where $\mathcal S$ is an integer stack and $N$ is a positive integer.
    \end{defin}

    The following proposition, originally from \cite[{Lemma 7.1.}]{Tomita15}. It has been restated in \cite{kp} with more specific statements. The latter version, in the notation of the current paper, reads as follows:

    \begin{lem}[Lemma 5.4., \cite{kp}]\label{stack}\label{Tomita15lemma7.1}
	Let $p$ be a selective ultrafilter, $m \in \omega$ and $(h_i: i <m)$ be sequences in $\mathbb Z^{(\mathfrak c)}$ which are $p$-independent mod constants (with respect to $G=\mathbb Z^{(\mathfrak c)}$).
 
 Then there exists $A\in p$, a positive integer $N$ and a integer stack $\calS$ on~$A$ 
	such that, for each $i<m$, $h_i|_A$  is an 
		integer combination of the elements of the sequences of the stack $\frac{1}{N} \calS$ restricted to $A$. 
		On the other hand, each sequence of the integer stack~$\calS$ is 
		an integer combination of $(h_i|A)_{i<m}$

	We will say in this case that the finite sequence 
	$( h_0, \ldots , h_{m-1})$ is $p$-associated to $(\calS, N, A)$.
	
\end{lem}

	It is worth mentioning that the hypothesis ``$(h_i: i<m)$ are sequences in $\mathbb Z^{(\mathfrak c)}$ which are $p$-independent mod constants'', in \ref{Tomita15lemma7.1}, reads as ``$(h_i:i<m)$  are sequences in $\mathbb Z^{(\mathfrak c)}$ such that the family $([h_i]_p:i<m)$ united with $([\vec \chi_{\alpha}]_p: \alpha<\mathfrak c)$ is linearly independent over $\mathbb Q$''. However, both these hypothesis are easily seen to be equivalent.

    We have the following:

    \begin{remark}\label{suptRemark}
        In the notation of Lemma \ref{Tomita15lemma7.1}, if $(h_i)_{i<m}$ is $p$-associated to $(\mathcal S, N, A)$, then  $\bigcup_{n \in A, i<m}\supp h_i(n)$ is the union of the supports of the elements of the integer stack $\mathcal S$. Also, if $M|N$, $(\mathcal S, M, A)$ is also associated to $(h_i)_{i<m}$.
    \end{remark}

    The importance of stacks with respect to countably compactness comes from the following result:

    \begin{lem}[Lemma 8.4, \cite{Tomita15}] \label{Tomita15lema8.4}
    Let  $(p_m : m \in \omega)$ be a family of incomparable selective ultrafilters and $( N_m : m \in \omega)$ be an enumeration of the positive integers such that each such integer is appears infinitely often.
    
    Let $\mathcal S_m$ be an integer stack on $A_m$ with $A_m \in p_m$ for each $m \in \omega$. In addition, consider $s^m$, $t^m$, $r^m_i$, $r^m_{i,j}$, $M^m$, $r^m$, $L^m$, $f^m_{i,j,k}$ and $g^m_l$ as part of the structure of $\mathcal S_m$.

    For any family $( c)^\frown(c^m_{i,j,k} : m < \omega, i < s^m, j < r^m_i, k < r^m_{i,j})^\frown (d^m_l \mid m < \omega, l < t^m)$ of $\mathbb Z^{(\mathfrak c)}$ with $c \neq 0$,  there exists a homomorphism $\phi : \mathbb{Q}^{(\mathfrak c)} \rightarrow \mathbb{\mathbb T}$ such that
        \begin{enumerate}
            \item[a)] $\phi(c) \neq 0$,
            \item[b)] $p_{m}- \lim\limits_{n \to \infty} \phi \left( \frac{1}{N_m} f^m_{i,j,k}(n) \right) = \phi \left(  c^m_{i,j,k} \right)$ for any $i < s^m$, $j < r^m_i$, $k < r^m_{i,j}$ and $m \in \omega$,
            \item[c)] $p_{m}- \lim\limits_{n \to \infty} \phi \left( \frac{1}{N_m} g^m_{l}(n) \right) = \phi ( d^m_{l} )$ for any $l < t^m$ and $m \in \omega$.
        \end{enumerate}
    \end{lem}

   We will need the following definition related to stacks:

    \begin{defin} 
    For $\gamma=0, 1$, let $\mathcal S_\gamma$, be the stacks defined by $s^\gamma$, $t^\gamma$, $M^\gamma$, $(r_i^\gamma: i<s^\gamma)$, $(r_{i, j}^\gamma: 0\leq i<s^\gamma, 0\leq j<r_i^\gamma)$, $(f_{i, j, k}^{\gamma}: i<s^\gamma, j<r_i^\gamma, k<r_{i, j}^\gamma)$, $(\xi_i^\gamma: i<s^\gamma$, $(\mu_l^\gamma: l<t^\gamma)$ and $(\theta_{i, j, k}^\gamma: i<s^\gamma, j<r_i^\gamma, k<r_{i, j}^\gamma)$.
    
    Assume $\mathcal S_0$ and $\mathcal S_1$ have disjoint supports. The union of $\mathcal S_0$, $\mathcal S_1$ is the stack $\mathcal S_0\sqcup \mathcal S_1$ defined as follows:

    Let $t=t^0+t^1$, $s=s^0+s^1$ and $M=M^0+M^1$. Let $g_l=g^0_l$ for $l<t_o$ and $g_l=g^1_{l-t_0}$ if $t^0\leq l <t^0+t^1$. If $i<s^0$ define $r_i=r^0_i$, $r_{i,j}=r^0_{i,j}$ for $j<r_i$ and $f_{i,j,k}=f^0_{i,j,k}$.

    If $s^0\leq i< s^0+s^1$ define $r_i=r^1_{i-s^0}$, $r_{i,j}=r^1_{i-s^0,j}$ for $j<r_i$ and $r_{i,j}=r^1_{i-s^0,j}$ for $j< r_{i}=r^1_{i-s^0}$ and $f_{i,j,k}=f^1_{i-s^0,j,k}$ for $j<r_i$ and $k<r_{i,j}$. Define $\xi_l$'s and $\theta_{i,j,k}$ simillarly.
    \end{defin}

    It is straightforward to verify that $\mathcal S_0\sqcup \mathcal S_1$ is an integer stack. We note that this easy verification is the only verification in this paper for which looking at bullet points i)-x) is needed, and we leave the proof as an easy exercise. Observe, however, that the hypothesis of disjoint supports is relevant. 

\section{Main result}

   In this section, we will show that $\mathbb N\oplus \mathbb Z^{(\mathfrak c)}$ can be endowed with a countably compact Hausdorff topology. Notationally, it will be easier to work with the isomorphic semigroup $\mathbb H=\{x \in \mathbb Z^{(\mathfrak c)}: x(0)\geq 0\}$.

   We will rely on the following easy lemma.

   \begin{lem}\label{freeSubgroup}
       Let $G$ be a torsion-free Abelian group. Then:       
       \begin{enumerate}
           \item $\Ult_p(G)/\underline G$ is torsion-free.
           \item For every finite subset $B$ of $\Ult_p(G)$, $(\langle B\rangle \oplus \underline G)/\underline G$ is free and finitely-generated
       \end{enumerate}
   \end{lem}
   \begin{proof}

   $\Ult_p(G)/\underline G$ is torsion-free: if $a\in \Ult_p(G)$ and $n$ is a positive integer such that $n(a+\underline G)=0$, then $na \in \underline G$. Thus, there exists $g \in G$ such that $na=n[\vec g]_p=[\vec{ng}]_p$. Write $a=[f]_p$. Thus, $\{i \in \omega: n(f(i)-g)=0\}\in p$, so $\{i \in \omega: f(i)-g=0\}\in p$, entailing $a=[f]_p=[\vec g]_p\in \mathbb G$.
   
    For the second item, $(\langle B\rangle \oplus \underline G)/\underline G$ is torsion-free by item (1). As it is also finitely generated, it follows that it is free.
   \end{proof}

    For the next proof, recall that a finitely generated torsion-free Abelian group is free. 
    \begin{prop} \label{PropWw}
    Let $( r_i \mid i < m)$ be a finite family in $\mathbb N^\omega$ and $p$ be an ultrafilter. There exist a family $(h_i: i<k)$ in $\mathbb Z^{\omega}$, a family $(b_i: i<k) \in \mathbb N^k$ and a family $(a_{i, j}: i<m, j<k)$ of integers such that:
    
    \begin{enumerate}
        \item $[r_i]_p = [\vec b_i + \displaystyle\sum_{j < k} a_{i,j}.h_j]_p$ for every $i<k$, and
        \item $([h_j]_p: j<k)$ is $p$-independent mod constants (with $G=\mathbb Z$).
    \end{enumerate}
    \end{prop}

    \begin{proof}   
    Let $B=\{[r_i]_p: i<m\}$. By Lemma \ref{freeSubgroup}, $H=(\langle B\rangle\oplus \underline {\mathbb Z})/\underline {\mathbb Z}$ is finitely-generated and free. Let $( g_j : j < k)$ be a family of members of $\mathbb{Z}^{\omega}$ such that $( [g_j]_p +\underline{\mathbb Z} : j < k )$ is a basis of  $H$. Thereby, $([g_j]_p: j < k )^\frown ([\vec 1]_p)$ is a basis of $\langle B\rangle\oplus \underline{\mathbb Z}$.
    
    So, for every $i < m$, there are $z_i \in \mathbb Z$ and $(a_{i,j}: j < k ) \in  \mathbb{Z} ^{k}$ such that $[r_i] = \left[\vec z_i+\sum_{j < k} a_{i,j}.g_j \right]_p$. We have that $A=\bigcap_{j<k}\{n \in \omega:  r_i(n)=z_i+\sum_{j<k}a_{i, j}g_j(n)\}\in p$. Let $n_0\in A$ and, for each $i<m$, let  $b_i=r_i(n_0)\in \mathbb N$ and $h_j=g_j-\vec{g_j(n_0)}$.
    
    It follows that for every $n \in A$ and $i<n$:

    $$b_i+\sum_{j<k}a_{i, j}h_j(n)=z_i+\sum_{j<k}a_{i, j}g_j(n_0)+\sum_{j<k}a_{i, j}[g_j(n)-g_j(n_0)]=z_i+\sum_{j<k}a_{i, j}g_j(n)=h_i(n).$$
    \end{proof}

    For the remaining of this paper, we will assume the existence of $\mathfrak c$ incomparable selective ultrafilters. First, we enumerate all finite subsets of sequences in $\mathbb H$ and then enumerate all the associated stacks that will be necessary by using the following notation:
    
    \begin{notation} \label{wstack} We fix the following:
    \begin{itemize}
    \item $( p_{\alpha} : \alpha \in \Lim )$ a family of pairwise incomparable selective ultrafilters.
        \item $\mathbb H=\{x \in \mathbb Z^{(\mathfrak c)}: x(0)\geq 0\}=\mathbb N. \chi_0 \oplus \mathbb{Z}^{(\mathfrak{c} \setminus \lbrace 0 \rbrace)}$.
        \item $\mathbb H_0=\{x \in \mathbb H: \supp x\subseteq \{0\}\}\approx \mathbb N$ and $\mathbb H_1=\{x \in \mathbb H: 0 \notin \supp x\}\approx \mathbb Z^{(\mathfrak c)}$.

        \item $( (h^{\alpha}_{0}, ..., h^{\alpha}_{m_{\alpha}-1}) : \alpha \in \Lim  )$ a surjective enumeration of all nonempty finite families of sequences in $\mathbb H$ such that $\displaystyle\bigcup_{i < m_{\alpha}} \supp h^{\alpha}_{i} \subset \alpha$ for every $\alpha \in \Lim$.
                \end{itemize}
        For each $\alpha \in \Lim$, we define:
        \begin{itemize}
        \item $( u^{\alpha}_i : i < m_{\alpha}) $ and $( v^{\alpha}_i : i < m_{\alpha})$ families of sequences in  $\mathbb{H}_1$ and $\mathbb N$ (respectively) such that $h^{\alpha}_i = v^{\alpha}_i. \vec\chi_{ {0}} + u^{\alpha}_i$ for each $i < m_{\alpha}$.
        \item $( w^{\alpha}_j : j < k_{\alpha})$ a family of sequences in $\mathbb Z^{\omega}$, $(b_i^\alpha: i<m_{\alpha})\in \mathbb N^{m_\alpha}$ $(a_{i, j}^\alpha: i<m_\alpha, j<k_\alpha)$ is a family of integers such that \begin{enumerate}
            \item $\forall i<m_\alpha\, [v_i^\alpha]_{p_\alpha}=[\vec b_i^\alpha+\sum_{j<k_\alpha} a_{i, j}^\alpha w_j^\alpha]_{p_\alpha}.$
            \item  $([h_i]_{p_\alpha}: i<m_\alpha)$ is $p_\alpha$-independent mod constants (with respect to $\mathbb Z$).

        \end{enumerate}

        This exists by Proposition \ref{PropWw}. Notice that the second item implies that          $([h_i\chi_0]_{p_\alpha}: i<m_\alpha)$ is $p_\alpha$-independent mod constants with respect to $\mathbb Z^{(\mathfrak c)}$.
        \item $(\mathcal S^0_{\alpha}, N^0_{\alpha}, C^0_\alpha)$  a stack $p_\alpha$-associated to $(w_j^\alpha \chi_0: j<k_\alpha)$. This exists by Proposition \ref{Tomita15lemma7.1}. Notice that the supports of the sequences in $\mathcal S_\alpha^0$ are contained in $\{0\}$.
        \item  $(z^{\alpha}_i : i<o_\alpha)$ a finite family of sequences in $\mathbb Z^{(\mathfrak c)}$ such that 
        $([z^\alpha_i]_{p_\alpha}+ \underline{\mathbb Z^{(\mathfrak c)}}: i<o_\alpha)$ 
        is a basis of $(\langle [u_i^\alpha]_{p_\alpha}: i<m_\alpha\rangle \oplus \underline{\mathbb Z^{(\mathfrak c)}})/\underline{\mathbb Z^{(\mathfrak c)}}$.

        This exists by Proposition \ref{freeSubgroup}. Notice that $(z^{\alpha}_i : i<o_\alpha)$ is a $p_\alpha$-independent mod constants family. We may choose the $z_i^\alpha$'s so that $0 \notin \supp z^i_\alpha$.

                \item $(\mathcal S^1_{\alpha}, N^1_{\alpha}, C^1_\alpha)$ a stack $p_\alpha$-associated to $(z_i^\alpha: i<o_\alpha)$.
                
                This exists by Proposition \ref{Tomita15lemma7.1}. By Remark \ref{suptRemark}, $0$ is not in the supports of the elements of $\mathcal S^1_\alpha$.
                \item $A_{\alpha} = \bigcap_{i<m_\alpha}\{n \in C^0_{\alpha}  \cap C^1_{\alpha}: v_i^{\alpha} b_i^\alpha+\sum_{j<k_{\alpha}}a_{i, j}^\alpha w_j^\alpha(n)\}\in p_{\alpha}$.
                \item $N_\alpha=N_\alpha^1.N_\alpha^2$, $\mathcal S_\alpha=\mathcal S_\alpha^0\sqcup \mathcal S_\alpha^1$.
                \item An injective, finite enumeration of the sequences of $\mathcal S_\alpha$, $(y^\alpha_j: j<q_\alpha)$.
    \end{itemize}

    Notice that  $( w^{\alpha}_j \chi_0:j<k_\alpha)^\frown ( z^{\alpha}_i \mid i < o_{\alpha} )$ is $p_\alpha$-associated to $(\mathcal S_{\alpha}, N_{\alpha}, A_{\alpha})$.

    Furthermore, the supports of the elements of $\mathcal S^\alpha$ are contained in $\bigcup_{i < m_{\alpha}} \supp h^{\alpha}_{i} \subset \alpha$ for every $\alpha \in \Lim$.         
    \end{notation}

    Finally, we give a positive answer to Question 9.4 of \cite{Tomita15}.
    
    \begin{exam}\label{maintresult}
   Assuming the existence of $\mathfrak c$ pairwise incomparable selective ultrafilters. There exists a group topology on $\mathbb Z^{(\mathfrak c)}$ for which all the finite powers of the subsemigroups $\mathbb H$ and $\mathbb H_1$ are countably compact.
    \end{exam}

    \begin{proof}
    Fix $c \in \mathbb Z^{(\mathfrak c)}\setminus \{0\}$. There exists a countable set $T_c \subset \mathfrak c$  such that
        \begin{itemize}
            \item $\supp c \subset T_{c}$, 
            \item if $\alpha \in T_{c}$ and $\beta \leq \alpha < \beta + \omega$ then $[\beta, \beta + \omega) \subset T_{c}$,
            \item if $\alpha \in T_{c}$ is a limit ordinal then $\bigcup_{j<q_\alpha} \supp y^\alpha_j \subset T_{c}$,
            \item for each positive integer $N$, there exists infinitely many $\alpha \in T_{c}$ such that $N_{\alpha} = N$.       
        \end{itemize}

    Let $( \alpha_m \mid m \in \omega )$ be an enumeration of all limit ordinals in $T_c$. In virtue of Lemma \ref{Tomita15lema8.4}, there exists a homomorphism $\phi_c : \mathbb{Q}^{(T_c)} \rightarrow \mathbb{T}$ such that    
        \begin{enumerate}
            \item[a)] $\phi_c(c) \neq 0$,
            \item[b)] $p_{\alpha_m}- \lim\limits_{n \to \infty} \phi_c \left( \frac{1}{N_{\alpha_m}} y^{\alpha_m}_{j}(n) \right) = \phi_c \left( \chi_{\alpha_m + j} \right)$ for all $m \in \omega$, $j< q_{\alpha_m}$.
        
        \end{enumerate}

    By the divisibility of $\mathbb T$ and compactness, we can recursively extend $\phi_c$ to $\mathbb Q^{(\mathfrak c)}$ so that:
        \begin{enumerate}
            \item[a')] $\phi_c(c) \neq 0$,
            \item[b')] $p_{\alpha}- \lim\limits_{n \to \infty} \phi_c \left( \frac{1}{N_{\alpha}} y^{\alpha}_{j}(n) \right) = \phi_c \left( \chi_{\alpha + j} \right)$ for all $m \in \omega$, $j< q_\alpha$.
        
        \end{enumerate}

    Thus, the weak topology on $\mathbb Q^{(\mathfrak c)}$ generated by the family of homomorphisms $(\phi_c : c \in \mathbb Z^{(\mathfrak c)} \setminus \lbrace 0 \rbrace )$, which is a Hausdorff group topology, satisfies:
            \begin{enumerate}
            \item[a'')] $\phi_c(c) \neq 0$,
            \item[b'')] $p_{\alpha}- \lim\limits_{n \to \infty} \frac{1}{N_{\alpha}} y^{\alpha}_{j}(n)= \ \chi_{\alpha + j} $ for all $m \in \omega$, $j< q_\alpha$.
        
        \end{enumerate}
    
    Therefore, for each $\alpha \in \Lim$, the sequences in the stack $ \frac{1}{N_\alpha}\mathcal S_\alpha$ have $p_\alpha$-limits  in $\{\chi_{\alpha+j}: j<q_\alpha\}$. Thus the sequences in $( w^{\alpha}_j \chi_0:j<k_\alpha)$ and $( z^{\alpha}_i : i < o_{\alpha} )$  have a $p_\alpha$-limits in $\langle \chi_{\alpha+j}:j<q_\alpha \rangle \subseteq 
    \mathbb H_1$. By the definition of the $z^\alpha_i$'s, it follows that each $u^\alpha_i$'s has a $p_\alpha$-limit in $\mathbb H_1$ as well. This proves that all powers of $\mathbb H_1^m$ are countably compact.
    
    By the definition of the $w^\alpha_j$´s and $b^\alpha_i$'s it then follows that the $p_\alpha$-limit of each $v_\alpha^i\chi_0$ is in $\mathbb H$, thus, the $p_\alpha$-limit of each $u_i^\alpha$ is in $\mathbb H$ as well. This proves that all powers of $\mathbb H^m$ are countably compact.
      \end{proof}

\section{$\mathbb H$ has no semigroup topology which makes its $\omega$-th power countably compact}

In this section, we prove (in ZFC) that $\mathbb N\oplus \mathbb Z^{(\mathfrak c)}$ does not admit a Hausdorff group topology which makes its $\omega$-th power countably compact, concluding the proof of the consistency of the existence of a Wallace group whose $\omega$-th power is its first power that is not countably compact. 

We prove something more general. For that purpose, we say that an Archimedian semigroup with identity is a $4$-uple $(T, +, 0, \leq )$ such that $(T, +, 0)$ is an Abelian semigroup with identity $0$, $\leq $ is a total order such that for all $a, b, c, d \in A$, if $a\leq b$ and $c\leq  d$, then $a+c\leq b+d$, and for all $a>0$ and $b\in T$, there exists $n \in \mathbb N$ such that $na>b$. An element $a \in T$ is said to be non-negative if $a\geq 0$.

In what follows, $G$ needs not to be Abelian.
\begin{prop}\label{restrictionZ}
    Let $G$ be a semigroup with identity and $T$ be an Archimedian semigroup with identity whose all elements are non-negative. If $T\neq \{0\}$, then there is no Hausdorff semigroup topology in $T\times G$ which makes its $\omega$-th power countably compact.
\end{prop}

\begin{proof} Assume by contradiction that $H=T\times G$ admits such a topology.

    Fix any $z \in H$ and $\bar c\in T\setminus \{0\}$. Write $c=(\bar c, e_G)$, and for every $n\in \omega$, let $nc=(n\bar c, e_G)$.

    For each $k \in \omega$, let: $$x_k=(z)_{i<k}\,^\frown(nc)_{n \in \omega}.$$
    As $H^\omega$ is countably compact, there exists a free ultrafilter $p$ such that for every $k\in \omega$, there exists $a_k=p$-$\lim x_k$.
    
    By continuity, for every $k \in \omega$, $a_{k}+kc=a_0$.

    There exists $r \in T$ and $g \in G$ such that $a_0=(r, g)$. Let $k\in \omega$ be such that $k\bar c> r$. Write $a_k=(r', g')$. Then $(r, g)=a_0=a_{k}+kc= (r'+k\bar c, g')$. 

    Thus, $r'+k\bar c=r$. As $r'\geq 0$, we have that $r=r'+k\bar c\geq k\bar c+0=k\bar c>r$, a contradiction.
\end{proof}

Notice that all subsemigroups of $\mathbb R$ of non-negative elements (in particular, $\mathbb N$) satisfy what is required by $T$. Thus, combining the preceding result with Example \ref{maintresult} we get:
\begin{cor}\label{mainexample}
    Assuming the existence of $\mathfrak c$ pairwise incomparable selective ultrafilters, there exists a topology in $\mathbb N\oplus \mathbb Z^{(\mathfrak c)}$ which makes it a Wallace semigroup whose first power that fails to be countably compact is $\omega$.
\end{cor}

\section{Comments and questions}
	
    With Corollary \ref{mainexample} in mind, it is natural to ask:
    
    \begin{question}
    Given an uncountable cardinal $\kappa\leq 2^\mathfrak c$, is there a Wallace semigroup $S$ for which $\kappa$ is the smallest cardinal such that $S^\kappa$ is not countably compact? In particular, is there such a semigroup for $\kappa=\omega_1$, $\kappa=\mathfrak c$ or $\kappa=2^{\mathfrak c}$?
    \end{question}

    Due to the restrictions imposed by \cite[Corollary 12, Corollary 14]{T98} and Proposition \ref{restrictionZ}, a natural candidate for the question above is a subsemigroup of $\mathbb Q ^{(\mathfrak c)}$. We note that Bellini, Rodrigues and Tomita  \cite{T21} showed that for every selective ultrafilter $p$, this group admits a $p$-compact group topology, unlike the free Abelian group which cannot have its $\omega$-th power countably compact
    \cite{T98}.

    In \cite{Boero}, a Wallace semigroup whose square is countably compact was obtained directly as the subsemigroup $\mathbb N^{(\mathfrak c)}$ with the subspace topology. Naturally, we ask:

    \begin{question}
    Does $\mathbb N^{(\mathfrak c)}$ admits a semigroup topology that makes its finite powers countably compact? Whose third power is countably compact?
    \end{question}
\section{Acknowledgements}
The first author has received financial support from CAPES (Brazil) as PhD student at the University of São Paulo under supervision of the third listed author.

The second author is a voluntary (unpaid) postdoctoral researcher at the University of São Paulo. He thanks the University for the hospitality. He thanks the financial support of his parents Alberto Rodrigues, Tânia de Oliveira and his wife Bruna Nagano.

The third author has received financial support from FAPESP 2021/00177-4.

	\bibliographystyle{amsplain}

\begin{thebibliography}{1}

     \bibitem{kp} M. K. Bellini, K. P. Hart, V. O. Rodrigues and A. H. Tomita {\em Countably compact group topologies on arbitrarily large free Abelian groups}, Topology and its Applications, 333 (2023), 108538.

    \bibitem{T21} M. K. Bellini, V. O. Rodrigues and A. H. Tomita, {\em On countably compact group topologies without non-trivial convergent sequences on $\mathbb{Q}^{(\kappa)}$ for arbitrarily large $\kappa$ and a selective ultrafilter}, Topology and its Applications, 294 (2021), 107653.
	   
    \bibitem{Boero} A. C. Boero and A. H. Tomita, {\em A group topology on the free abelian group of cardinality $\mathfrak{c}$ that makes its square countably compact}, Fundamenta Mathematicae, 3.212 (2011), 235-260.
    
        
    \bibitem{Grant} D. L. Grant, {\em Sequentially compact cancellative topological semigroups: some progress on the Wallace problem}, Annals of the New York Academy of Sciences, 704.1 (1993), 150-154.
    
    \bibitem{RS} D. Robbie and S. Svetlichny, {\em An answer to A. D. Wallace’s question about countably compact cancellative semigroups}, Proceedings of the American Mathematical Society, 124.1 (1996), 325-330. 
    
    \bibitem{gelbaum1951embedding}B. Gelbaum and G. K. Kalisch and J. M. H. Olmsted, JMH, {\em On the embedding of topological semigroups and integral domains}, Proceedings of the American Mathematical Society, 2.5 (1951), 807-821.
    
   \bibitem{T96} A. H. Tomita, {\em The Wallace Problem: a counterexample from $MA_{countable}$ and $p$-compactness}, Canadian Mathematical Bulletin, 39.4 (1996), 486-498.
       
    \bibitem{T98} A. H. Tomita, {\em The existence of initially $\omega_1$-compact group topologies on free Abelian groups is independent of ZFC}, Commentationes Mathematicae Universitatis Carolinae, 39.2 (1998), 401-413.

    
    \bibitem{Tomita15} A. H. Tomita, \emph{A group topology on the free abelian group of cardinality ${\mathfrak c}$ that makes its finite powers countably compact}, Topology and its Applications, 196 (2015), 976-998.

    
    \bibitem{W} A. D. Wallace, {\em The structure of topological semigroups}, Bulletin of the American Mathematical Society, 61.2 (1955), 95-112.

\end{thebibliography}

\end{document}